\documentclass{amsart}

\usepackage{amssymb, amsmath, amsthm, braket}

\theoremstyle{plain}
  \newtheorem{thm}{Theorem}[section]
  \newtheorem{defi}[thm]{Definition}
  
  \newtheorem{prop}[thm]{Proposition}
  \newtheorem{cor}[thm]{Corollary}
  \newtheorem{lemma}[thm]{Lemma}
\theoremstyle{definition}
  \newtheorem{ex}[thm]{Example}
  \newtheorem{rem}[thm]{Remark}
  
\newcommand{\R}{\mathbb{R}}
\newcommand{\PP}{\mathbb{P}}
\newcommand{\C}{\mathbb{C}}
\newcommand{\K}{\mathbb{R}}
\newcommand{\p}{\partial}
\newcommand{\dimens}{\mathrm{dim}}
\newcommand{\ev}{\mathit{ev}}
\newcommand{\mV}{\mathcal{V}}
\newcommand{\bpm}{\begin{pmatrix}}
\newcommand{\epm}{\end{pmatrix}}
\newcommand{\skp}[2]{{\Braket{{#1},{#2}}}}

\begin{document}

\title{Jointly orthogonal polynomials}
\author{Giovanni Felder}
\address{Department of mathematics,
ETH Zurich, 8092 Zurich, Switzerland}
\email{felder@math.ethz.ch}

\author{Thomas Willwacher}
\address{Department of mathematics,
ETH Zurich, 8092 Zurich, Switzerland}
\email{thomas.willwacher@math.ethz.ch}

\thanks{The authors are supported in part by the Swiss National Science Foundation, Grant PDAMP2\_137151.}

\begin{abstract}
The theory of polynomials orthogonal with respect to one inner product is classical. We discuss the extension of this theory to multiple inner products.
Examples include the Lam\'e and Heine-Stieltjes polynomials.
\end{abstract}
\maketitle

\section{Introduction}
Let $\skp{}{}$ be an inner product on the space of polynomials $\R[x]$.
By the Gram-Schmidt process we may obtain a basis $E_0,E_1,\dots$ of $\R[x]$ orthogonal with respect to $\skp{}{}$. This basis, which we call a Gram-Schmidt basis, has the defining property that $E_n$ is of degree $n$ and is orthogonal to the space $V_{n-1}\subset \R[x]$ of polynomials of degree at most $n-1$.
In the most relevant cases the inner product has the form 
\[
\skp{f}{g} = \int_I f(x)g(x) w(x) dx,
\]
where $I\subset \R$ is a (possibly unbounded) interval, and $w(x)$ is an integrable almost everywhere positive weight function on $I$ for which all moments exist.

The classical orthogonal polynomials provide examples of Gram-Schmidt bases, which are in addition solutions of a polynomial differential equation $Lf=\lambda f$.
Concretely, for a bounded interval $I=(-1,1)$ and $w(x)=(1-x)^\alpha(1+x)^\beta$ one obtains the Jacobi polynomials, with $Lf=(1-x^2)f''+(\beta-\alpha-(\alpha+\beta+2)x)f'$. For a semi-bounded interval $I=(0,\infty)$ and $w(x)=x^\alpha e^{-x}$ one obtains the (associated) Laguerre polynomials, with $Lf=xf''+(\alpha+1-x)f'$.
For $I=(-\infty, \infty)$, $w(x)=e^{-x^2/2}$ one obtains the Hermite polynomials, with $Lf=f''-2xf'$.

In this paper we consider the generalization of the Gram-Schmidt process to $k\geq 1$ inner products 
\begin{equation}\label{equ:innerprodj}
 \skp{f}{g}_j = \int_{I_j} f(x)g(x) w_j(x) dx,
\end{equation}
where $j=1,\dots, k$ and $I_1,\dots ,I_k \subset \R$ are disjoint intervals and the $w_j$ are integrable, almost everywhere positive weight functions, for which all moments exist. 
In this situation it is not immediately obvious how to generalize the defining property $E_n\in V_{n-1}^\perp$ of the Gram-Schmidt basis. It turns out that a good generalization is what we call a \emph{jointly orthogonal system}.
To define it let us introduce the following notation. We denote by $V_n\subset \R[x]$ the space of polynomials of degree at most $n$. Let $V^k=S^k\R[x]$ be the space of symmetric polynomials in $k$ variables. Let $V_n^k=S^kV_n\subset V^k$ be the subspace of polynomials for which the degreee in each variable does not exceed $n$. For a polynomial $p\in \R[x]$ we let 
\begin{equation}\label{equ:prtimes}
p^{(r)} := \underbrace{p\otimes \cdots \otimes p}_{r \times} \in V^r.
\end{equation}
In particular we set $p^{(0)}:=1\in V^0\cong\R$. We call symmetric polynomials of the form $p^{(r)}$ \emph{rank one tensors}.

Assuming that the intervals $I_1,\dots, I_k$ are ordered from left to right we define an inner product $\skp{}{}$ on $V^k$ as follows:
\begin{multline}\label{equ:kipro}
 \skp{f}{g} = \idotsint_{I_1\times \cdots \times I_k} f(x_1,\dots, x_k)g(x_1,\dots, x_k)
\\
w_1(x_1)\cdots w_k(x_k)
\prod_{1\leq i<i'\leq k}(x_{i'}-x_i)
dx_1\cdots dx_k.
\end{multline}
Note in particular that the Vandermonde factor is always positive by the assumption on disjointness and orderedness on the intervals.
Finally, we define $k$ inner products $\skp{}{}_{(1)},\dots,\skp{}{}_{(k)}$ on $V^{k-1}$ such that
\begin{multline}\label{equ:k1ipro}
  \skp{f}{g}_{(j)} = \idotsint_{I_1\times \cdots \hat I_j \cdots \times I_k} f(x_1,\dots, \hat x_j, \dots, x_k)g(x_1,\dots,\hat x_j, \dots, x_k)
\\
w_1(x_1)\cdots \hat w_j(x_j) \cdots w_k(x_k)
\prod_{\substack{1\leq i<i'\leq k\\ i,i'\neq j }}(x_{i'}-x_i)
dx_1\cdots \widehat{dx_j} \cdots dx_k.
\end{multline}
where a hat denotes omission. 

\begin{defi}\label{def:jointorthog}
 A jointly orthogonal system of degree $n$ with respect to inner products $\skp{}{}_1,\dots, \skp{}{}_k$ ($k\geq 2$) as above is a family of polynomials $\{E_\alpha\}_\alpha$ of degree $n$ in one variable such that the family $\{E_\alpha^{(k-1)}\}_\alpha$ is a basis of $V_n^{k-1}$ orthogonal with respect to each of the inner products $\skp{}{}_{(1)},\dots,\skp{}{}_{(k)}$.
\end{defi}

Note that for $k=2$ a jointly orthogonal system of degree $n$ is just a basis of $V_n$ that is simultaneously orthogonal with respect to the two inner products $\skp{}{}_1,\skp{}{}_2$. 

Our main result is the following.

\begin{thm}\label{thm:main}
 Let $\skp{}{}_1,\dots ,\skp{}{}_k$ be inner products on $\R[x]$ of the form \eqref{equ:innerprodj} for disjoint intervals $I_1,\dots, I_k$. Then for every $n=0,1,\dots$ a jointly orthogonal system of degree $n$ with respect to these inner products exists, and it is unique up to rescaling and permutation of its members.
\end{thm}

The following result relates the notion of joint orthogonality to the Gram-Schmidt property.

\begin{prop}\label{prop:firstsecondid}
 A family of polynomials $\{E_\alpha\}_\alpha$ is a jointly orthogonal system of degree $n$ with respect to inner products $\skp{}{}_1,\dots, \skp{}{}_k$ as above if and only if the family of polynomials $\{E_\alpha^{(k)}\}_\alpha$ forms an orthogonal basis of $(V_{n-1}^k)^\perp\subset V_n^k$, where the orthogonal complement is taken with respect to the inner product $\skp{}{}$ as in \eqref{equ:kipro}. 
\end{prop}


In the literature, jointly orthogonal systems arise as solutions to many well studied polynomial differential equations. 
For example, for $k=2$ it is well known that the Lam\' e polynomials form jointly orthogonal systems. In this case the two intervals $I_1,I_2$ are bounded. The Ince polynomials form jointly orthogonal systems, with one of the intervals being unbounded. The polynomials arising in the doubly unbounded case have no special name to our knowledge but were studied by Turbiner, see also section \ref{sec:sexticpotential}.
For $k>2$ the known examples are provided by the Heine-Stieltjes polynomials or certain higher Heine-Stieltjes polynomials \cite{shapiro}.

\begin{rem}
Note that the uniqueness part of Theorem \ref{thm:main} allows to define each of these classical families of polynomials purely in terms of their orthogonality properties, without reference to any differential equation. 
\end{rem}

\subsection*{Structure of the paper}
In sections \ref{sec:examples} and \ref{sec:examples2} we will discuss several classical examples of jointly orthogonal systems given by solutions of polynomial differential equations. 

In section \ref{sec:jointorthog} we discuss some immediate properties of jointly orthogonal systems. In particular, we will prove the forward implication of Proposition \ref{prop:firstsecondid}.

In section \ref{sec:multieig} we show that the problem of finding a jointly orthogonal system can be restated as a multiparameter eigenvalue problem. Its solutions exist and are unique, hence Theorem \ref{thm:main} follows.

Finally section \ref{sec:gramschmidt} discusses our version of the Gram-Schmidt algorithm, which we call the rank 1 Gram-Schmidt algorithm. A uniqueness result therein will show the reverse implication of Proposition \ref{prop:firstsecondid}.

For most of the proofs it will be convenient to work in a slightly more general setting than discussed in the introduction.
Essentially, we allow general inner products, replacing the condition of disjointness of the intervals by a more abstract definiteness condition. The relevant definitions can be found in section \ref{sec:jointorthogdef}.

\subsection*{Acknowledgements}
We are grateful to Alexander Veselov for useful comments and instructive discussions on the subject of this paper.

\section{Examples: Jointly orthogonal polynomials for two inner products}\label{sec:examples}
Let us consider the case of two inner products, i.e., $k=2$. A jointly orthogonal system of degree $n$ is just a basis of $V_n$ simultaneously othogonal with respect to both inner products. 

\subsection{Heun and Lam\'e polynomials}
Let $Q(x)=(x-e_1)(x-e_2)(x-e_3)$, be a monic polynomial of degree $3$
with distinct real roots $e_1<e_2<e_3$ and $P(x)$ a polynomial of
degree $2$ such that $a_i:=P(e_i)/Q'(e_i)>0$. Set
\[
w_1(x)=w_2(x)=w(x)=\prod_{i=1}^3|x-e_i|^{a_i-1},\quad I_1=(e_1,e_2),\quad I_2=(e_2,e_3).
\]
The Heun differential operator, cf.~\cite{Erdelyetal},
\begin{align*}
L_n&=Q(x)\frac{d^2}{dx^2}+P(x)\frac{d}{dx}-n(n-1+a_1+a_2+a_3)x
\\
&=Q(x)\left(\frac{d^2}{dx^2}+\sum_{i=1}^3\frac{a_i}{x-e_i}\frac{d}{dx}\right)-n(n-1+a_1+a_2+a_3)x
\end{align*}
preserves $V_n$ and 
is self-adjoint with respect to both inner products. The latter
property follows by integration by parts by using the relation
\begin{equation}\label{e-logder}
 \frac{w'(x)}{w(x)}=\frac{P(x)-Q'(x)}{Q(x)}.
\end{equation}
Indeed this identity implies that
\[
\int_a^bL_nf(x)g(x)\,dx=\int_a^bf(x)L_ng(x)dx+(f'g-fg')Qw|_{a}^{b}
\]
and since 
\begin{equation}\label{e-boundary}
\lim_{x\to x_0} Q(x)w(x)=0, \text{  for all endpoints $x_0$ of the intervals $I_1,I_2$,}
\end{equation}
$L_n$ is selfadjoint for both inner products. Thus the jointly
orthogonal polynomials $E_{\alpha,n}$, $\alpha=1,\dots,n+1$, 
of degree $n$ are the polynomial eigenvectors of
the Heun differential operators.

Lam\'e polynomials arise in the special cases where $a_i\in\{1/2,3/2\}$.
By definition they are solutions of the Lam\'e differential equation
\[
Q(x)y''(x)+\frac12 Q'(x)y-\frac14(\nu(\nu+1)x+\lambda)y=0,
\]
with $Q$ as above, of the form 
\begin{equation}\label{e-Lamepol}
y(x)=\prod_{i=1}^3(x-e_i)^{\epsilon_i/2}p(x)
\end{equation}
with $p(x)$ a polynomial and $\epsilon_i\in\{0,1\}$. 
If $\epsilon_1=\epsilon_2=\epsilon_3=0$ the Lam\'e equation is the eigenvalue
problem for the Heun operator $L_n$
with parameters $n=\nu/2,a_i=\frac12$, $i=1,2,3$. The other cases can be
related to this case: let $\epsilon\in\{0,1\}^3$ and $L_n^{(\epsilon)}$ denote 
the Heun operator with parameters
$n$, $a_i=\frac12+\epsilon_i$. Then we have the
identity
\[
\prod_{i=1}^3(x-e_i)^{-\epsilon_i/2}\circ L_{n}^{(000)}\circ\prod_{i=1}^3(x-e_i)^{\epsilon_i/2}=L_{n-(\epsilon_1+\epsilon_2+\epsilon_3)/2}^{(\epsilon)}+c,
\]
for some multiple of the identity $c$.

Solutions of the Lam\'e equation of the form \eqref{e-Lamepol}
with $1+\epsilon_1+\epsilon_2+\epsilon_3=k$ are called Lam\'e
polynomials of degree $\nu$ of the $k$-th species. The parameter
$\nu$ dictates the behaviour at infinity of solutions and so the
degree $n$ of $p$:
\[
\nu = 2n+\epsilon_1+\epsilon_2+\epsilon_3.
\] 
Let
$E_{\alpha,n}^{(\epsilon)}, \alpha=1,\dots,n+1$ be the jointly orthogonal
 polynomials of degree $n$ with parameters $a_i=\epsilon_i+1/2$.
Then for even $\nu=2m\geq0$, there are Lam\'e polynomials of the
first species 
$
E_{\alpha,m}^{(000)}, \alpha=1,\dots,m+1,
$
and of the third species
\[
(x-e_1)^{\frac12}(x-e_2)^{\frac12}
E_{\alpha,m-1}^{(110)},\;
(x-e_1)^{\frac12}(x-e_3)^{\frac12}
E_{\alpha,m-1}^{(101)},\;
(x-e_2)^{\frac12}(x-e_3)^{\frac12}
E_{\alpha,m-1}^{(011)},
\]
$\alpha=1,\dots,m$.
For odd $\nu=2m+1\geq0$, there are Lam\'e polynomials of the fourth species
\[
(x-e_1)^{\frac12}(x-e_2)^{\frac12}(x-e_3)^{\frac12}E_{\alpha,m-1}^{(111)},
\] 
and of the second species
\[
(x-e_1)^{\frac12}E_{\alpha,m}^{(100)},
\;(x-e_1)^{\frac12}E_{\alpha,m}^{(010)},
\;(x-e_1)^{\frac12}E_{\alpha,m}^{(001)}.
\]
In both cases there is a total of $2\nu+1$ Lam\'e polynomials of degree $\nu$.

The Heun differential equation $L_ny=\lambda y$ is a Fuchsian differential equation
with four singular points $(e_1,e_2,e_3,\infty)$
on the Riemann spheres. The next examples are confluent
versions where singular points merge at infinity.
\subsection{Whittaker--Hill equation and Ince polynomials}
Let $\alpha<0, a>0$, 
\[
w(x)=e^{2\alpha x}|1-x|^{a_1-1}|1+x|^{a_2-1}, 
\quad I_1=(-1,1),\quad I_2=(1,\infty)
\]
For $a_i=\frac12,\frac32$, 
the jointly orthogonal polynomials are related
to solutions  of the {\em Whittaker--Hill equation}, see \cite{Whittaker, Ince, Arscott,NIST, HemeryVeselov},
\[
-\psi''(\theta)-(A\cos(2\theta)+B\cos(4\theta))\psi(\theta)=\lambda\psi(\theta)
\]
in trigonometric polynomials.
Let us explain the relation.
The differential operator
\[
L_n=(x^2-1)\frac{d^2}{dx^2}+(a_1(x+1)+a_2(x-1)+2\alpha(x^2-1))\frac d{dx}-2n\alpha x
\]
clearly preserves the space of polynomials of degree $\leq n$. It is self-adjoint
for both inner products as its coefficients $Q(x)=x^2-1$ and $P(x)=a_1(x+1)+a_2(x-1)+2\alpha(x^2-1)$
obey the relations \eqref{e-logder}, \eqref{e-boundary}.
The change of variables $x=\cos(2\theta)$ transforms $L_n$ to the differential operator
\[
\tilde L_n=-\frac14\frac {d^2}{d\theta^2}+
\left(-\alpha\sin(2\theta)+\frac{a_1+a_2-1}2\cot(2\theta)+\frac{a_1-a_2}{2\sin(2\theta)}\right)
\frac d{d\theta}-2n\alpha\cos(2\theta).
\]
Let $a_1,a_2\in\{\frac12,\frac32\}$ and set $\epsilon_i=a_i-\frac12\in\{0,1\}$. Let
\[
\varphi(\theta)=\sin(\theta)^{\epsilon_1}\cos(\theta)^{\epsilon_2}e^{\alpha\cos(2\theta)}.
\]
Then
$H_\nu=\varphi\circ 4\tilde L_n\circ \varphi^{-1}=H_\nu+c$ is the
Whittaker--Hill differential operator
\[
H_\nu=-\frac {d^2}{d\theta^2}-4\alpha\nu\cos(2\theta)-2\alpha^2\cos(4\theta),
\]
with parameter $\nu=2n+1+\epsilon_1+\epsilon_2$, up to an additive
constant
$c=2\alpha^2+4(\epsilon_2-\epsilon_1)\alpha-(\epsilon_1+\epsilon_2)^2$.
It follows that the jointly orthogonal polynomials
$E^{(\epsilon_1\epsilon_2)}_{1,n},\dots,E^{(\epsilon_1\epsilon_2)}_{n+1,n}$
for $a_i=\epsilon_i+\frac12\in\{\frac12,\frac32\}$ yield
eigenfunctions of $H_\nu$ of the form
\begin{equation}\label{e-Ince}
\psi_\alpha^{(\epsilon_1\epsilon_2)}(\theta)=
\sin(\theta)^{\epsilon_1}
\cos(\theta)^{\epsilon_2}
E_{\alpha,n}^{(\epsilon_1\epsilon_2)}
(\cos(2\theta))e^{\alpha\cos(2\theta)},
\quad \nu=2n+1+\epsilon_1+\epsilon_2.
\end{equation}
Altogether we get for each positive integer $\nu$ 
a basis of the space of eigenfunctions of $H_\nu$
of the form 
\[
p(\theta)e^{\alpha\cos(2\theta)},
\qquad 
p\in \mathit{TP}_{\nu-1}=\mathbb R[\cos(\theta),\sin(\theta)]_{\leq \nu-1},
\]
the space of trigonometric polynomials of
degree $\leq \nu-1$. If $\nu$ is even, the trigonometric
polynomials are
\[
\sin(\theta)E^{(10)}_{\alpha,\nu/2-1}(\cos(2\theta)),\; 
\cos(\theta)E^{(01)}_{\alpha,\nu/2-1}(\cos(2\theta)),\quad \alpha=1,\dots,\nu/2,
\]
and they form a basis of the subspace of $\mathit{TP}_{\nu-1}$
of $\pi$-antiperiodic polynomials, i.e., 
obeying $p(\theta+\pi)=-p(\theta)$.
If $\nu$ is odd they are
\begin{align*}
E^{(00)}_{\alpha,(\nu-1)/2}(\cos(2\theta)),\quad \alpha=1,\dots,(\nu+1)/2,
\\  
\sin(\theta)\cos(\theta)E^{(11)}_{\alpha,(\nu-3)/2}(\cos(2\theta)),\quad \alpha=1,\dots,(\nu-1)/2,
\end{align*}
and they are a basis of the subspace of 
$\pi$-periodic polynomials in $\mathit{TP}_{\nu+1}$.
The trigonometric polynomials arising this way are called
Ince
polynomials. They are a solutions of the differential equation
\[
p''+4\alpha\sin(2\theta)p'+(4(\nu-1)\alpha\cos(2\theta)+\lambda)p=0
\]
for different values of the spectral parameter $\lambda$.

\subsection{Eigenfunctions of Schr\"odinger operators with
sextic potential}\label{sec:sexticpotential}
Let $\ell>-1$,
\[
w(x)=|x|^{\ell+1/2} e^{-x^2/2}dx,\quad I_1=(-\infty,0),\quad I_2=(0,\infty).
\]
The differential operator
\[
L=-2x\frac{d^2}{dx^2}+(2x^2-2\ell-3)\frac d{dx}-2nx
\]
preserves polynomials of degree $\leq n$ and is 
self-adjoint for both inner products, since its coefficients
$Q(x)=-2x$ and $P(x)=2x^2-2\ell-3$ obey \eqref{e-logder} and \eqref{e-boundary}.
After the change of variables $x=r^2$, $L$ becomes
\[
\tilde L=-\frac12 \frac {d^2}{dr^2}+\left(r^3-\frac{\ell+1}{r}\right)\frac d{dr}-2n r^2
\]
This operator is related by conjugation by
\[
\varphi(r)=r^{\ell+1}e^{-r^4/4}
\]
to
the radial Schr\"odinger operator with sextic potential
\[
H=-\varphi\circ 2\tilde L\circ\varphi^{-1}
=-\frac{d^2}{dr^2}+r^6-\nu r^2+\frac{\ell(\ell+1)}{r^2},
\]
with $\nu=4n+5+2\ell$, introduced by Turbiner \cite{Turbiner} in his theory of quasi-exactly solvable
systems and is a special case of the family of monodromy free Schr\"odinger operators with 
sextic growth \cite{GibbonsVeselov}.
It follows that the simultaneous orthogonal polynomials $E_1,\dots,E_{n+1}$
give a basis 
\[
\psi_\alpha(r)=r^{\ell+1}e^{-r^4/4}E_\alpha(r^2)
\]
of eigenfunctions of $H$ in 
\[
W_n=\{r^{\ell+1}e^{-r^4/4}p(r^2),\, p\in V_n\}.
\]
The simultaneous orthogonality translates to the condition
that these eigenfunctions are characterized as being 
jointly orthogonal for the two inner products on $W_n$:
\[
\int_0^\infty f(r)g(r)dr,\qquad e^{-i\pi(\ell+1)}
\int_0^\infty f\left(e^{i\pi/2}r\right)
g\left(e^{i\pi/2}r\right)dr.
\]

\section{Examples for higher $k$: Heine-Stieltjes polynomials}\label{sec:examples2}

Consider the differential equation
\begin{equation}\label{equ:HS}
\prod_{i=0}^k(x-e_i) y'' + \sum_{j=0}^k m_j \prod_{\substack{i=0 \\ i\neq j}}^k(x-e_i)y' + V(x) y = 0.
\end{equation}
Here the unknowns are $y$ and the degree $k-1$ polynomial $V(x)$ (the van Vleck polynomial), which generalizes the eigenvalue $\lambda$ in the previous examples. 
The Heine-Stieltjes polynomials are the polynomial solutions $y$ of the above equation.
For each van Vleck polynomial $V$ there is at most one polynomial solution $y$. Furthermore note that for Heine-Stieltjes polynomials of degree $n$ the leading coefficient of the van Vleck polynomial must necessarily be $-n(n-1+\sum_j m_j)$
We assume in the following that the roots $e_i$ are real and ordered, $e_0<e_1<\cdots <e_k$, and that $m_0,\dots, m_k>0$. 
The differential operator
\[
L= \prod_{i=0}^k(x-e_i) \frac{d^2}{dx^2} + \sum_{j=0}^k m_j \prod_{\substack{i=0 \\ i\neq j}}^k(x-e_i)\frac{d}{dx} 
\]
is self-adjoint with respect to the $k$ inner products 
\[
\skp{f}{g}_j = \int_{e_{j-1}}^{e_j}
f(x)g(x) w(x) dx
\]
where 
\[
w(x)= \prod_{i=0}^k |x-e_i|^{m_j-1}.
\]
There is a well-known orthogonality relation satisfied by Heine-Stieltjes polynomials, found by Germanski \cite{germanski} and rediscovered by Volkmer \cite{volkmer}:
\begin{equation}
\label{equ:HSortho}
 \skp{E_\alpha^{(k)}}{E_\beta^{(k)}} = 0
\end{equation}
where $E_\alpha, E_\beta$ are two distinct (i.e., non-proportional) Heine-Stieltjes polynomials and the inner product is the inner product \eqref{equ:kipro} on $V^k$. 
To see \eqref{equ:HSortho}, note that 
\[
 \skp{E_\alpha}{(V_\alpha - V_\beta) E_\beta}_j =\skp{E_\alpha}{ L E_\beta}_j - \skp{L E_\alpha}{ E_\beta}_j =0
\]
where $V_\alpha$, $V_\beta$ are the (necessarily distinct) van Vleck polynomials corresponding to $E_\alpha, E_\beta$.
It follows that the matrix 
\[
\bpm 
\skp{E_\alpha}{E_\beta}_1 & \cdots & \skp{E_\alpha}{E_\beta}_k \\
\vdots & \ddots & \vdots \\
\skp{E_\alpha}{x^{k-1} E_\beta}_1 & \cdots & \skp{E_\alpha}{x^{k-1}E_\beta}_k 
\epm
\]
annihilates the coefficient vector of $V_\alpha - V_\beta$. Hence the determinant of the above matrix is zero. This is easily calculated to be the statement \eqref{equ:HSortho}.

Note furthermore that if $E_\alpha, E_\beta$ are both of degree $n$, then $V_\alpha - V_\beta$ is of degree $\leq k-1$ since the leading coefficients of the van Vleck polynomials are necessarily the same. Hence the matrix
\[
\bpm 
\skp{E_\alpha}{E_\beta}_1 & \cdots & \skp{E_\alpha}{E_\beta}_k \\
\vdots & \ddots & \vdots \\
\skp{E_\alpha}{x^{k-2} E_\beta}_1 & \cdots & \skp{E_\alpha}{x^{k-2}E_\beta}_k 
\epm
\]
must be rank deficient and hence all its $k$ minors vanish. The vanishing of the $j$-th minor is easily checked to be equivalent to the condition
\[
 \skp{E_\alpha^{(k-1)}}{E_\beta^{(k-1)}}_{(j)} = 0
\]
where $\skp{}{}_{(j)}$ is the scalar product \eqref{equ:k1ipro} on $V^{k-1}$.

It is furthermore well known that the Heine-Stieltjes differential equation has $N={ n+k-1 \choose k-1 }$ polynomial solutions of degree $n$, in one-to-one correspondence with the $N$ ways of distributing the $n$ roots of $y$ among the $k$ intervals $(e_{j-1},e_j)$ \cite{Stieltjes}. 
It follows that the Heine-Stieltjes polynomials of degree $n$ form a jointly orthogonal system of degree $n$ according to Definition \ref{def:jointorthog}.



\section{Jointly orthogonal systems}\label{sec:jointorthog}
\subsection{Definitions}\label{sec:jointorthogdef}
As in the introduction we denote the space of symmetric polynomials in $k$ variables by $V^k$ and let $V_n^k \subset V^k$ be the space of symmetric polynomials of degree at most $n$ in any variable. For example $x_1^3x_2^2+x_2^3x_1^2\in V_3^2$, while $x_1^4x_2+x_2^4x_1\notin V_3^2$. 

Suppose we are given $k$ inner products $\skp{}{}_1,\dots, \skp{}{}_k$ on the space of polynomials $\K[x]$. We suppose throughout this paper and without further mention the following condition:

\vskip .3cm

{\bf Standing Assumption:} The operator of multiplication by $x$ is self-adjoint for all scalar products on $\K[x]$ considered.

\vskip .3cm

The $k$ inner products on $\K[x]$ can be used to define an symmetric bilinear form $\skp{}{}$ on $V^k$ such that
\begin{equation}
 \label{equ:bilink}
 \skp{p^{(k)}}{q^{(k)}}
=
\det
\bpm
\skp{p}{q}_1 & \cdots & \skp{p}{q}_k \\
\skp{p}{xq}_1 & \cdots & \skp{p}{xq}_k \\
\vdots & \ddots & \vdots \\
\skp{p}{x^{k-1}q}_1 & \cdots & \skp{p}{x^{k-1}q}_k
\epm
\end{equation}
for all $p,q\in \K[x]$, where we used again the notation \eqref{equ:prtimes}. 
Depending on the inner products $\skp{}{}_j$, this bilinear form may or not be definite. We will generally assume the following condition.

\vskip .3cm

{\bf Definiteness Condition 1:} For all non-zero $p\in \K[x]$: $\skp{p^{(k)}}{p^{(k)}}\neq 0$.

\vskip .3cm


\begin{rem}
 Note that by renumbering the scalar products $\skp{}{}_j$ we may change the sign of $\skp{}{}$, so we might as well ask that $\skp{p^{(k)}}{p^{(k)}}>0$ in the above condition.
\end{rem}

Similarly, we define $k$ symmetric bilinear forms $\skp{}{}_{(1)}, \dots, \skp{}{}_{(k)}$ on $V^{k-1}$ such that
\begin{multline}\label{equ:bilink1}
 \skp{p^{(k-1)}}{q^{(k-1)}}_{(j)}
= \\
\det
\bpm
\skp{p}{q}_1 & \cdots & \skp{p}{q}_{j-1} & \skp{p}{q}_{j+1}  & \cdots & \skp{p}{q}_k \\
\vdots & \ddots & \vdots & \vdots & \ddots & \vdots \\
\skp{p}{x^{k-2}q}_1 & \cdots & \skp{p}{x^{k-2}q}_{j-1} &\skp{p}{x^{k-2}q}_{j+1} & \cdots & \skp{p}{x^{k-2}q}_k
\epm
\end{multline}
for all $p,q\in \K[x]$. These bilinear forms again may or may not be definite. However, for some results below we will assume the following condition:

\vskip .3cm

{\bf Definiteness Condition 2:} For all non-zero $p\in \K[x]$ and all $j=1,\dots,k$: $\skp{p^{(k-1)}}{p^{(k-1)}}_j\neq 0$.

\vskip .3cm


Assuming that Definiteness Conditions 1 and 2 are satisfied, we extend Definition \ref{def:jointorthog} of jointly orthogonal systems verbatim to this more general setting.

\begin{ex}
Scalar products $\skp{}{}_j$ as in \eqref{equ:innerprodj} for disjoint intervals $I_1,\dots ,I_k$ satisfy the Definiteness Conditions 1 and 2. 
To see this, note that the bilinear form \eqref{equ:bilink} takes on the form
\begin{equation}\label{equ:kproduct}
\langle f,f \rangle 
=
 \idotsint |f(x_1, \dots , x_k)|^2 \prod_{1\leq i< j\leq k}(x_i-x_j)  \prod_{j=1}^k w_j(x_j) dx_j 
\end{equation}
where $f\in \K[x_1,\dots, x_k]$, using the Vandermonde formula.
Since the intervals are disjoint, the Vandermonde factor is non-zero and has a definite sign, while all other terms are non-negative. Hence Definiteness Condition 1 holds. By an analogous argument, Definiteness Condition 2 also holds.
\end{ex}

\subsection{Properties}\label{sec:properties}

\begin{lemma}\label{lem:firstimpliessecond}
Let $\{E_\alpha\}_\alpha$ be a jointly orthogonal system of degree $n$ with respect to the inner products $\skp{}{}_1,\dots,\skp{}{}_k$. Then the symmetric polynomials $E_\alpha^{(k)}$ are orthogonal (with respect to the bilinear form \eqref{equ:bilink}) to all symmetric polynomials $Q\in V_{n-1}^k\subset V_n^k$. 
\end{lemma}
\begin{proof}
Note that for any polynomial $p$ of degree at most $n$ and for all $i,j\in \{1,\dots,k\}$,
\[
 \sum_\alpha \frac{\langle E_\alpha^{(k-1)} ,p^{(k-1)} \rangle_{(i)}}{\langle  E_\alpha^{(k-1)} ,E_\alpha^{(k-1)} \rangle_{(i)}} E_\alpha^{(k-1)}
=p^{(k-1)}
=
\sum_\alpha \frac{\langle E_\alpha^{(k-1)} ,p^{(k-1)} \rangle_{(j)}}{\langle  E_\alpha^{(k-1)} ,E_\alpha^{(k-1)} \rangle_{(j)}} E_\alpha^{(k-1)}.
\]
Hence
\begin{equation}\label{equ:coeffequal}
 \frac{\langle E_\alpha^{(k-1)} ,p^{(k-1)} \rangle_{(i)}}{\langle  E_\alpha^{(k-1)} ,E_\alpha^{(k-1)} \rangle_{(i)}}
=
\frac{\langle E_\alpha^{(k-1)} ,p^{(k-1)} \rangle_{(j)}}{\langle  E_\alpha^{(k-1)} ,E_\alpha^{(k-1)} \rangle_{(j)}}
\end{equation}
for any polynomial $p$ of degree $\leq n$ and any $\alpha$.

Our goal is to show that for any $Q\in V_{n-1}^k$ and any $\alpha$
\[
\skp{E_\alpha}{Q}=0.
\]
Since the rank one tensors span $V_{n-1}^k$ it suffices to show the above equation for $Q$ of the form $q^{(k)}$, where $q$ is a polynomial of degree at most $n-1$. We will distinguish two cases: 

(i) Suppose $\langle E_\alpha^{(k-1)} ,q^{(k-1)} \rangle_{(1)}=0$. Then by \eqref{equ:coeffequal}, $\langle E_\alpha^{(k-1)} ,q^{(k-1)} \rangle_{(i)}=0$ for all $i$. Hence
\[
\langle E_\alpha^{(k)} ,q^{(k)} \rangle
=
\sum_{i=1}^k (-1)^{i+k}
\langle E_\alpha ,x^{k-1} q \rangle_{i}
\langle E_\alpha^{(k-1)} ,q^{(k-1)} \rangle_{(i)}
= 0,
\]
expanding the determinant \eqref{equ:bilink} with respect to the last row.

(ii) Suppose $\langle E_\alpha^{(k-1)} ,q^{(k-1)} \rangle_{(1)}\neq 0$. Expand the determinant \eqref{equ:bilink} with respect to the first row and compute

\begin{align*}
\langle E_\alpha^{(k)} ,q^{(k)} \rangle
&=
\sum_{i=1}^k (-1)^{i+1}
\langle E_\alpha ,q \rangle_{i}
\langle E_\alpha^{(k-1)} ,(xq)^{(k-1)} \rangle_{(i)}
\\
&=
\sum_{i=1}^k (-1)^{i+1}
\langle E_\alpha ,q \rangle_{i}
\langle E_\alpha^{(k-1)} ,(xq)^{(k-1)} \rangle_{(1)}
\frac{\langle E_\alpha^{(k-1)} ,E_\alpha^{(k-1)} \rangle_{(i)}}{\langle E_\alpha^{(k-1)} ,E_\alpha^{(k-1)} \rangle_{(1)}}
\\
&=
\frac{\langle E_\alpha^{(k-1)} ,(xq)^{(k-1)} \rangle_{(1)}}{\langle E_\alpha^{(k-1)} ,E_\alpha^{(k-1)} \rangle_{(1)}}
\sum_{i=1}^k (-1)^{i+1}
\langle E_\alpha ,q \rangle_{i}
\langle E_\alpha^{(k-1)} ,E_\alpha^{(k-1)} \rangle_{(i)}
\\
&=
\frac{\langle E_\alpha^{(k-1)} ,(xq)^{(k-1)} \rangle_{(1)}}{\langle E_\alpha^{(k-1)} ,q^{(k-1)} \rangle_{(1)}}
\sum_{i=1}^k (-1)^{i+1}
\langle E_\alpha ,q \rangle_{i}
\langle E_\alpha^{(k-1)} ,q^{(k-1)} \rangle_{(i)}
\\
&=0.
\end{align*}
Here we used \eqref{equ:coeffequal} twice, once for for $p=xq$ (this is possible since $q$ is of degree $\leq n-1$) and once for $p=q$. The last equality is true since the left hand side is a determinant of a matrix with two equal rows.
\end{proof}

An important feature of a jointly orthogonal system $E_\alpha$ is that it defines 
a family of orthogonal symmetric polynomials in $V_{n}^{k}$. 
\begin{lemma}\label{lem:northog}
Let $\{E_\alpha\}_\alpha$ be a family of polynomials such that the $E_\alpha^{(k-1)}$ form a jointly orthogonal basis of $V_{n}^{k-1}$ for the inner products $\skp{}{}_{(i)}$, $i=1,\dots,k$.
Then the vectors 
\[
 E^{(k)}_\alpha \in V_{n}^{k}
\]
are pairwise orthogonal with respect to \eqref{equ:bilink}. 
\end{lemma}

\begin{proof}
 As in \cite{germanski, volkmer} define, for each fixed pair $\alpha\neq \beta$, the $k\times k$ matrices 
\[
 M_{\alpha\beta} 
=
(
\langle x^{i-1} E_\alpha ,E_\beta \rangle_{j}
)_{ij}.
\]
Orthogonality of the $E^{(k)}_\alpha$ is the statement that 
\[
\det M_{\alpha\beta}=\langle E_\alpha^{(k)} ,E_\beta^{(k)} \rangle=0. 
\]
Let $M_{\alpha\beta}^j$ be the minor obtained by deleting the $k$-th row and the $j$-th column.
Joint orthogonality is the statement that 
\[
\det M_{\alpha\beta}^j=\langle E_\alpha^{(k-1)} ,E_\beta^{(k-1)} \rangle_{(j)}=0 
\]
for $j=1,2,\dots, k$.
Clearly, by expanding $\det M_{\alpha\beta}$ along the $k$-th row the statement of the Lemma follows.
\end{proof}

\begin{rem}
 Note that if we are given jointly orthogonal systems for each degree $n=0,1,\dots$, the symmetric polynomials $E_\alpha^{(k)}$ form an orthogonal basis of the space of symmetric polynomials $V^k$. By Theorem \ref{thm:main} this basis is canonical, i.e., uniquely defined (up to rescaling) and independent of arbitrary choices. This is in contrast to other methods of obtaining an orthogonal basis of $V^k$, for example by applying the Gram-Schmidt algorithm to an arbitrary non-orthogonal basis.
\end{rem}

%
%

\section{Joint orthogonality and multiparameter eigenvalue problems }\label{sec:multieig}

\subsection{Symmetric rectangular multiparameter eigenvalue problems}\label{sec:symmmultpar}
\hfill \\
Let $A_1,\dots, A_k$ be complex $(m+k-2)\times m$ matrices. A rectangular multiparameter eigenvalue problem is the equation
\begin{equation}
 \label{equ:smultpar}
 \sum_{i=1}^k \lambda_i A_i v = 0
\end{equation}
in complex unknowns $(\lambda_1,\dots, \lambda_k)\neq 0$ (the multiparameter eigenvalue, defined up to a multiplicative constant) and $0\neq v\in \C^n$ (the eigenvector).
Eqn. \eqref{equ:smultpar} is overdetermined and may have solutions or not.

Define the $m\times m$ matrix $A_{ij}$, $1\leq i \leq k$, $1\leq j \leq k-1$, as the submatrix of $A_i$ composed of the rows $j,\dots, j+m-1$. We call the rectangular multiparameter eigenvalue problem 
\emph{symmetric} if all matrices $A_{ij}$ are real and symmetric.

\begin{rem}
 For $k>2$ the matrices $A_i$ appearing in a symmetric rectangular multiparameter eigenvalue problem are necessarily Hankel matrices.
\end{rem}

 Similar to \eqref{equ:bilink1} define the $k$-vector valued sesquilinear form $\mu$ on $S^{k-1}\C^n$ such that
\[
 \mu(\underbrace{u\otimes \cdots \otimes u}_{k-1\times},\underbrace{v\otimes \cdots \otimes v}_{k-1\times})
:=
\vec{\det} (\bar u^T A_{ji} v )_{ij}.
\]
Here the vector valued determinant $\vec{\det}$ of a $(k-1)\times k$ matrix is the $k$-vector of its minors of size $k-1$, taken with alternating signs. 
If for all non-zero $v\in \C^n$: $\mu(v^{\otimes k-1},v^{\otimes k-1})\neq 0$ one calls the mutiparameter eigenvalue problem locally definite. 

In the following, we will say that two multiparameter eigenvalues are \emph{distinct} if they are not collinear. We will say that a multiparameter eigenvalue is \emph{real} if some non-zero multiple is real.

\begin{prop}\label{prop:eigval}
 Suppose $v$ is an eigenvector of a symmetric and locally definite rectangular multiparameter eigenvalue problem. Then the corresponding multiparameter eigenvalue is $\lambda=\mu(v^{\otimes k-1},v^{\otimes k-1})$ and is real. Suppose $u$ is another eigenvector corresponding to the eigenvalue $\lambda'$ which is distinct from $\lambda$. Then 
\[
 \mu(u^{\otimes k-1},v^{\otimes k-1})=0
\]
In addition all multiparameter eigenvectors may be taken real.
\end{prop}
\begin{proof}
 It is an easy verification, using Cramer's rule.
\end{proof}

For $k=2$ and $A_1$ the identity matrix, the proposition reduces to well known statements of elementary linear algebra.

\begin{rem}
We use here the notation ``rectangular multiparameter eigenvalue problem'' to distinguish it from a multiparameter eigenvalue problem in the sense of \cite{volkmer}. A rectangular multiparameter eigenvalue problem determines a multiparameter eigenvalue problem in the sense of loc. cit., and eigenvalues in our sense are eigenvalues in the sense of loc. cit. Since we only consider multiparameter eigenvalue problems that are rectangular in this paper, we will sometimes drop the adjective rectangular, abusing notation slightly.
\end{rem}

\begin{rem}
 A more general definition of symmetry of a rectangular multiparameter eigenvalue problem is provided in Appendix \ref{sec:symmmultpar2}.
\end{rem}

\subsection{The relation to joint orthogonality}
One motivation for Definition \ref{def:jointorthog} is the following Theorem. 

\begin{thm}\label{thm:multieig}
Assume inner products $\skp{}{}_1,\dots \skp{}{}_k$ on $\R[x]$ are given that satisfy the Definiteness Conditions 1 and 2. Then the following two conditions are equivalent for a family $\{E_\alpha\}_{\alpha=1,\dots, N}$ of degree $\leq n$ polynomials, with $N={ n+k-1\choose k-1 }=\dimens V_n^{k-1}$. 
\begin{itemize}
 \item $\{E_\alpha\}_\alpha$ is a jointly orthogonal system of degree $n$ with respect to the inner products given.
 \item The $E_\alpha$ are solutions to the symmetric rectangular multiparameter eigenvalue problem
\begin{equation}
\label{equ:eigval}
 \sum_{j=1}^k \Ket{E_\alpha}_j \lambda_{\alpha, j} =0
\end{equation}
for \emph{distinct} non-zero eigenvectors $\lambda_\alpha=(\lambda_{\alpha,1},\dots , \lambda_{\alpha,k})$, where $\Ket{E_\alpha}_j$ is the linear form $v\mapsto \skp{v}{E_\alpha}_j$ on the space of polynomials of degree $\leq n+k-2$.
\end{itemize}

In fact, all such $\lambda_\alpha$ are multiples of 
\begin{equation}
\label{equ:eigvalexpl}
 \left( \skp{E_\alpha^{(k-1)}}{E_\alpha^{(k-1)}}_{(1)}, \dots,(-1)^{j-1} \skp{E_\alpha^{(k-1)}}{E_\alpha^{(k-1)}}_{(j)} ,\dots, (-1)^{k-1} \skp{E_\alpha^{(k-1)}}{E_\alpha^{(k-1)}}_{(k)} \right).
\end{equation}
\end{thm}

\begin{rem}
In the usual monomial basis of the space of polynomials \eqref{equ:eigval} becomes a symmetric multiparameter eigenvalue problem of the form \eqref{equ:smultpar}. Here $m=n+1$ and $A_j$ is the matrix of the linear operator
\[
 V_n \to V_{n+k-2}^*
\]
sending $v\in V_n$ to $\ket{v}_j\in V_{n+k-2}^*$.
\end{rem}

\begin{cor}\label{cor:HSeig}
 The Heine-Stieltjes polynomials are solutions to the rectangular multiparameter eigenvalue problem above for distinct eigenvalues.
\end{cor}

\begin{proof}[Proof of the Theorem]
Suppose first that the $E_\alpha$ solve the eigenvalue problem with respect to distinct eigenvalues. Then by Proposition \ref{prop:eigval} the eigenvalues are obtained by \eqref{equ:eigvalexpl} and the $E_\alpha^{(k-1)}$ form an orthogonal basis of $V_n^{k-1}$ with respect to each of the bilinear forms $\skp{}{}_{(j)}$.

For the other direction, note that for all $p\in V_n$
\begin{align*}
0 &=
\sum_{j=1}^k (-1)^{j-1} \skp{ p }{E_\alpha}_j  \skp{ p^{(k-1)} }{E_\alpha^{(k-1)}}_{(j)} 
\\&=
  \frac{ \skp{ p^{(k-1)} }{E_\alpha^{(k-1)}}_{(1)} }{\skp{E_\alpha^{(k-1)}}{E_\alpha^{(k-1)}}_{(1)}}
  \sum_{j=1}^k  (-1)^{j-1} \skp{ p }{E_\alpha}_j  \skp{E_\alpha^{(k-1)}}{E_\alpha^{(k-1)}}_{(j)}.
 \end{align*}
 The first equation is the vanishing of a determinant with two equal rows, and for the second equation we used \eqref{equ:coeffequal}. Note that the right hand side is a product of a polynomial in (the coefficients of) $p$ and a linear function in $p$. If the product vanishes for all $p$, then one of the factors has to vanish identically. It can not be the first, since this factor is 1 for $p=E_\alpha$. Hence
\[
 \sum_{j=1}^k  (-1)^{j-1} \skp{ p }{E_\alpha}_j  \skp{E_\alpha^{(k-1)}}{E_\alpha^{(k-1)}}_{(j)} =0
\]
for all polynomials $p$ of degree at most $n$. 
Hence the left hand side of \eqref{equ:eigval} is zero on all polynomials of degree $\leq n$. Next, evaluate the left hand side of \eqref{equ:eigval} on $x^r E_\alpha$, $r=0,\dots, k-2$.
\[
  \sum_{j=1}^k (-1)^{j-1} \skp{E_\alpha}{x^r E_\alpha}_j \skp{E_\alpha^{(k-1)}}{E_\alpha^{(k-1)}}_{(j)}.
\]
This is however the determinant of a matrix with two equal rows and hence zero.
Summarizing, we have shown that the left hand side of \eqref{equ:eigval} vanishes on all polynomials of degrees $\leq n+k-2$. Hence the $E_\alpha$ are indeed multiparameter eigenvectors. It remains to be shown that the corresponding eigenvalues are distinct.
So suppose to the contrary that two members of the family both are eigenvectors for some eigenvalue $(\lambda_1,\dots, \lambda_k)$. Since the members of the family are non-collinear by assumption, the kernel of $\sum_{j=1}^k\lambda_j A_j$ has dimension $\geq 2$. It follows that there is polynomial of degree $<n$ in the kernel, say $v$. But then we may replace
\[
\skp{v}{x^r v}_1= -\sum_{j=2}^k \frac{\lambda_k}{\lambda_1} \skp{v}{x^r v}_j
\]
in the determinant defining $\skp{v^{(k)}}{v^{(k)}}$, see \eqref{equ:bilink}. The determinant of a matrix with linearly dependent columns vanishes, and hence arrive at a contradiction to Definiteness Condition 1. Here we used that $\lambda_1\neq 0$, which follows from \eqref{equ:eigvalexpl} and Definiteness Condition 2. Alternatively, since $\lambda \neq 0$ we may as well suppose that $\lambda_1\neq 0$. 
\end{proof}

\subsection{Existence of solutions, and the proof of Theorem \ref{thm:main}}
\begin{thm}
 Assume we are given $k$ inner products on the space of polynomials such that Definiteness Conditions 1 and 2 hold.
Then for every $n$ the eigenvalue problem \eqref{equ:eigval} has ${ n+k-1 \choose k-1 }$ distinct solutions $(\lambda_1,\dots, \lambda_k)\in \PP^{k-1}$. In particular it follows that there is a (unique up to rescaling) set of ${ n+k-1 \choose k-1 }$ multiparameter eigenvectors $E_\alpha$ that form a jointly orthogonal basis in the sense of Definition \ref{def:jointorthog}.
\end{thm}
\begin{proof}
 First let us show that there are $N:={ n+k-1 \choose k-1 }$ solutions, counted with multiplicity. For this we can use Shapiro's lemma, see Appendix \ref{sec:shapiro}.
Indeed, note that by the Definiteness Condition 2, there can be no eigenvalue $\lambda = (\lambda_1,\dots, \lambda_k)$ which has a zero component, and in particular none that has $\lambda_1=0$. It follows that $\lambda_2A_2+\dots+\lambda_k A_k$ is rank deficient iff all $\lambda_2=\lambda_3=\cdots=\lambda_k=0$. Hence the condition in Shapiro's Lemma is satisfied. Hence there are $N$ solutions counted with multiplicity. It remains to show that there are no solutions of multiplicity $>1$. 

First note that the kernel of $\sum_{j=1}^k\lambda_j A_j$ cannot have dimension $\geq 2$. Otherwise, we can arrive at contradiction to Definiteness Condition 1 as in the proof of Theorem \ref{thm:multieig}.

It remains to be shown that if $\lambda$ is an eigenvalue of multiplicity $\geq 2$, then the corresponding eigenspace is of dimension $\geq 2$.  The set of $A_j$'s for which there are multiple eigenvalues is Zariski closed, and there is a point in the complement by Corollary \ref{cor:HSeig}, or alternatively by the explicit computation of Appendix \ref{sec:exmultipar}. Hence we may always perturb the eigenvalue problem so as to lift the degeneracies of eigenvalues, so say we set $A_j^\epsilon= A_j+\epsilon B_j$. Then there are eigenvectors $v_\epsilon,w_\epsilon$ to the distinct perturbed eigenvalues, continuously depending on $\epsilon>0$, such that $v^{(k)}_\epsilon$, $w^{(k)}_\epsilon$ are orthogonal. To disregard the arbitrary multiplicative factor, we also consider the corresponding one-dimensional eigenspaces $V_\epsilon,W_\epsilon\in \PP^{n-1}$, which are uniquely defined. By compactness, they have to have limit points (as $\epsilon\to 0$), and there must be at least two different limit points by orthogonality of $V_\epsilon,W_\epsilon$. Suppose $v, w\neq 0$ are vectors in these limit points (subspaces). Then by continuity $\sum_j \lambda_j^\epsilon A_j^\epsilon$ must annihilate both $v$ and $w$ at $\epsilon=0$. Hence the kernel is at least $2$-dimensional.
\end{proof}


\section{Orthogonal symmetric polynomials and the rank 1 Gram-Schmidt algorithm}\label{sec:gramschmidt}

In this section we show that the two conditions of Definition \ref{def:jointorthog} are slightly stronger than 
necessary to guarantee uniqueness.

\begin{prop}\label{prop:rank1unique}
Assume an inner product on the space of symmetric polynomials $V^k$ is given.
Suppose we are given two families of polynomials $\{E_\alpha\}_\alpha$, $\{E_\alpha'\}_\alpha$ with $E_\alpha, E_\alpha'\in V_n$, such that the families
$\{E_\alpha^{(k)}\}_\alpha$ and $\{(E_\alpha')^{(k)}\}_\alpha$ both form bases of $(V_{n-1}^{k})^\perp\subset V_n^k$. Then the two families are identical, up to rescaling and permutation of its members.
\end{prop}

The proof will be given in the subsequent section.

\begin{rem}
It follows that the rank 1 basis of $(V_{n-1}^k)^\perp$ in this case is essentially the set of rank one, norm one tensors in $(V_{n-1}^k)^\perp$. ``Essentially'' here means that if $v$ is of rank one and norm one, so is $-v$ of course, so we must pick one of $v$ or $-v$ for the basis.

In particular, the jointly orthogonal system in the sense of Definition \ref{def:jointorthog} is essentially the set of polynomials $v\in V_n$ such that $v^{(r)}\in (V_{n-1}^k)^\perp$.
\end{rem}

The following algorithm can hence be used to determine the (essentially unique) jointly orthogonal systems of each degree. 

\vskip .3cm

{\bf Rank 1 Gram-Schmidt Algorithm:}
\begin{enumerate}
 \item Initialization: A jointly orthogonal system of degree 0 is given by the constant polynomial $E_{0,1}:=1\in V_0$. Set $n=1$.
 \item Solve the system of homogeneous polynomial equations 
\[
 \skp{v^{(k)}}{E_{n-1,\alpha}^{(k)}}=0
\]
for $v\in V_n$ where $\alpha=1,2,\dots , { n+k-2 \choose k-1 }$.
Assuming that the inner product is of the form \eqref{equ:bilink} and assuming Definiteness Conditions 1 and 2 one finds ${ n+k-1 \choose k-1 }$ solutions up to rescaling. They become the jointly orthogonal system $\{E_{n,\alpha}\}_\alpha$ of degree $n$.
 \item Increase $n$ and go to step 2.
\end{enumerate}

\vskip .3cm

\subsection{Proof of Proposition \ref{prop:rank1unique}}

\begin{rem}\label{rem:basis}
 Note that a (non-orthogonal) basis of $V^k_n$ may be given by the monomial symmetric polynomials $m_\mu$ where $\mu=(\mu_1,\dots, \mu_k)$, $n\geq \mu_1\geq \mu_2 \geq \cdots \geq \mu_k \geq 0$ is a multiindex. Counting such multiindices, one sees that 
\[
 \dim V_n^k = \bpm n+k \\ k \epm.
\]
It follows that 
\[
\dim (V_{n-1}^k)^\perp = \bpm n+k \\ k \epm - \bpm n+k-1 \\ k \epm = \bpm n+k-1 \\ k-1 \epm = \dim V^{k-1}_n.
\]
Explicitly, a map $V^{k-1}_n\to (V^k_{n-1})^\perp$ may be defined
as the composition $V^{k-1}_n\to V^k_n \to (V^k_{n-1})^\perp$ where the first arrow sends 
\[
 m_{\mu_1, \dots, \mu_{k-1}} \mapsto  m_{n, \mu_1, \dots, \mu_{k-1}}
\]
and the second is the orthogonal projection.
An explicit isomorphism in the other direction $(V^k_{n-1})^\perp\to V^{k-1}_n$ is given by the operator $\frac{\p^n}{\p x_k^n}$. In particular, the images of the elements $E_\alpha^{(k)}$ of a rank one basis under this operator are non-zero multiples of 
\[
E_\alpha^{(k-1)} = E_\alpha(x_1)E_\alpha(x_2)\cdots E_\alpha(x_{k-1}).
\]
Hence it follows that the $E_\alpha^{(k-1)}$ form a basis of $V^{k-1}_n$.
\end{rem}

\begin{proof}[Proof of Proposition \ref{prop:rank1unique}]
Suppose there is some degree $n$ polynomial $p\in \C[x]$ such that $p^{(k)}:= p(x_1)\cdots p(x_k)$ is in  $(V^k_{n-1})^\perp$. We want to show that $p^{(k)}$ is a multiple of some $E_\alpha^{(k)}$. Since the $E_\alpha^{(k)}$ form a basis, we may write 
\[
 p^{(k)} = \sum_\alpha \lambda_\alpha E_\alpha^{(k)}
\]
for some constants $\lambda_\alpha$.
Let $a_1,\dots, a_r$ be the roots of $p$, with multiplicities $m_1,\dots, m_r$.
Let $\ev_{x_k=a}$ be the operator of evaluation at $x_k=a$ and apply the operators 
\[
 \ev_{x_k=a_i}  \frac{\p^j}{\p x_k^j}
\]
on both sides of the above equation for $i=1,2,\dots, r$, $j=0,1,\dots, m_r-1$.
We obtain $n$ equations of the form 
\[
 0 = \sum_\alpha \lambda_\alpha c_{\alpha ij} E_\alpha^{(k-1)}.
\]
for some constants $c_{\alpha ij}$ which are zero iff $E_\alpha$ has a root $a_i$ with multiplicity not equal to $j$. Since by Remark \ref{rem:basis} above the $E_\alpha^{(k-1)}$ are elements of a basis and hence linearly independent, we must have $\lambda_\alpha c_{\alpha ij}=0$ for each $\alpha, i, j$. Hence $\lambda_\alpha=0$ unless $E_\alpha$ has the same roots, with the same multiplicities, as $p$. But then $p\propto E_\alpha$.
\end{proof}

%


Note that by Theorem \ref{thm:main} the Rank 1 Gram-Schmidt algorithm will succeed in finding the jointly orthogonal systems if the inner product on $V^k$ has the form \eqref{equ:bilink} and Definiteness Conditions 1 and 2 are satisfied. For a general inner product however, the algorithm might fail, as the following example shows.

\begin{ex}
For a general inner product on $V^k$ a rank 1 basis may or may not exist.
As an example, consider the $n=1$, $k=2$ case. Polynomials in this case may be identified with $2\times 2$ symmetric matrices. Fix a basis $e_1=x_1x_2$, $e_2=x_1+x_2$, $e_3=1$. Consider the inner product
\begin{align*}
\langle e_1 , e_1 \rangle &= \langle e_3 , e_3 \rangle = 1
&
\langle e_2 , e_2 \rangle &= \lambda 
\\
\langle e_1 , e_2 \rangle &= \langle e_2 , e_3 \rangle = 0
&
\langle e_1 , e_3 \rangle &= -\epsilon^2
\end{align*}
where $\lambda>0$, $0\leq \epsilon<1$.
The rank 1 vectors in the orthogonal complement of $e_3$ are $e_1\pm \epsilon e_2 + \epsilon^2 e_3$.
So in particular, there may be either 2 or only one, depending on $\epsilon$. 
Furthermore, these two vectors are orthogonal iff
\[
 1-\epsilon^4 - \lambda \epsilon^2 = 0,
\]
so while generically the vectors are not orthogonal, for a specific value of $\lambda$ they are.
\end{ex}


%
%
%
%

\subsection{Proof of Proposition \ref{prop:firstsecondid}}
The forward implication of Proposition \ref{prop:firstsecondid} is the content of Lemmas \ref{lem:firstimpliessecond} and \ref{lem:northog} of section \ref{sec:properties}. 

For the reverse implication note that the by Proposition \ref{prop:rank1unique} the condition of Proposition \ref{prop:firstsecondid} determines the family $\{E_\alpha\}_\alpha$ uniquely up to rescaling and permutation, if such a family exists. However, by Theorem \ref{thm:main} and the forward implication we know that at least one such family exists, namely a jointly orthogonal system. Hence, by uniqueness, the family $\{E_\alpha\}_\alpha$ must be a jointly orthogonal system.
\hfill \qed

\appendix

\section{Shapiro's Lemma}\label{sec:shapiro}

Let $\mV$ be the variety of rank deficient $n\times m$ matrices ($n>m$). It has been shown in \cite{brunsvetter} that the degree of the variety $\mV$ is ${ n \choose m-1 }$.

\begin{lemma}[Shapiro \cite{shapiro}]
Let $A, A_1, \dots, A_{n-m-1}$ be $n\times m$ matrices such that the subspace $W$ spanned by $A_1,\dots, A_{n-m-1}$ intersects $\mV$ only at zero. Then the variety $A+W$ intersects $\mV$ at exactly ${ n \choose m-1}$ points, counted with multiplicity.
\end{lemma}
\begin{proof}
 First, intersections at the infinite plane correspond to rank deficient linear combinations of $A_1,\dots, A_{n-m-1}$. Hence by assumptions, there are no intersection loci on the infinite plane. Hence there can be no intersection loci of dimension $\geq 1$, since they would automatically contain infinite points. Hence all intersection loci are points. Bezout's Theorem then says that there are ${ n \choose m-1 }$ points, counted with multiplicity.
\end{proof}

\section{An example multiparameter eigenvalue problem}\label{sec:exmultipar}
Let us consider the following multiparameter eigenvalue problem:
\begin{equation}\label{equ:eigex}
 \bpm
\lambda_1 & 0 & 0 & \dots & 0 & 0 & \lambda_k \\
\lambda_2 & \lambda_1 & 0 & \dots & 0 & 0 & 0 \\
\lambda_3 & \lambda_2 & \lambda_1 & \dots & 0 & 0 & 0 \\
\lambda_4 & \lambda_3 & \lambda_2 & \dots & 0 & 0 & 0 \\
\vdots & & & \ddots & & & \vdots \\
\vdots & & & \ddots & & & \vdots \\
\vdots & & & \ddots & & & \vdots \\
0 & 0 & 0 & \dots & \lambda_{k} & \lambda_{k-1} & \lambda_{k-2}  \\
0 & 0 & 0 & \dots & 0 & \lambda_{k} & \lambda_{k-1} \\
\epm
v = 0
\end{equation}
where, as before, $\lambda=(\lambda_1, \dots, \lambda_k)$ is the multiparameter eigenvalue and $v$ is an $n+1$-vector, the multiparameter eigenvector. The eigenvalue problem is not symmetric as it stands, but may be put into symmetric form by reversing the order of the columns of the above matrix. The definiteness condition does not hold in this case in general.

We will show that the eigenvalue problem has a complete set of solutions, so there are ${n+\kappa \choose \kappa}$ distinct eigenvalues, where we set $\kappa:=k-1$.
They are given by the following construction:
\begin{enumerate}
 \item Pick $\kappa$ distinct $(n+\kappa)$-th roots of unity, call them $\zeta_1,\dots, \zeta_\kappa$. 
There are clearly ${n+\kappa \choose \kappa}$ choices, each of which will give us one eigenvalue and one eigenvector.
 \item Set $\lambda_j$ to be the coefficient of $x^{\kappa-j}$ in the polynomial $\prod_{i=1}^\kappa(x-\zeta_j)$. 
This defines the multiparameter eigenvalue.
Clearly, different choices of $\zeta_j$'s yield distinct eigenvalues.
 \item Let $v=(v_0,\dots, v_n)$. Then we set
\[
 v_j =
\det
\bpm
1 & 1  & \dots  & 1 \\
\bar \zeta_1 & \bar \zeta_2 &  \bar \dots &  \bar \zeta_\kappa \\
\bar \zeta_1^2 & \zeta_2^2 &  \dots &  \bar \zeta_\kappa^2 \\
\vdots &  & \ddots &  \vdots \\
\bar \zeta_1^{\kappa-1} & \bar \zeta_2^{\kappa-1} &  \dots &  \bar \zeta_\kappa^{\kappa-1} \\
\zeta_1^{j+1} & \zeta_2^{j+1} &  \dots &  \zeta_\kappa^{j+1} \\
\epm 
\]
We may expand along the last row to write this alternatively as 
\begin{equation}\label{equ:powersum}
 v_j = \sum_i p_i \zeta_i^{j+1}
\end{equation}
where the $p_i$ are the appropriate minors. They are Vandermonde determinants and independent of $j$.
\end{enumerate}

\begin{lemma}
 The above $\lambda$, $v$ solve the eigenvalue problem \eqref{equ:eigex}.
\end{lemma}
\begin{proof}
Componentwise, the eigenvalue problem are the equations 
\begin{equation}\label{equ:exrec}
 \sum_{i=1}^k \lambda_i v_{j+1-i}=0
\end{equation}
for $j=0,1,\dots, n+\kappa-1$, where we set 
\begin{align}\label{equ:bdry}
 v_{-1}&=\cdots=v_{-\kappa+1}=0
&
 v_{n+1}&=\cdots=v_{n+\kappa-1}=0 
&
 v_{-\kappa}&=v_n.
\end{align}
We call these latter equations the \emph{boundary conditions}.

It is well known that the Ansatz \eqref{equ:powersum} (for arbitrary $p_i$) solves the recursion \eqref{equ:exrec}, if the $\zeta_i$ are the roots of the characteristic polynomial 
\[
 \sum_{i=1}^k \lambda_i^{n-i} x^i.
\]
In our case they are by definition of $\lambda$. However, we need to verify that the boundary conditions \eqref{equ:bdry} are also respected by the Ansatz \eqref{equ:powersum}.
For the first set of equations in \eqref{equ:powersum} note that $\zeta_i^{j+1}=\bar \zeta_i^{-j-1}$, and hence the determinant vanishes for $j=-1,\dots, -\kappa+1$. Similarly, since the $\zeta_i$ are $(n+\kappa)$-th roots of unity, the same holds for $j=n+1, \dots, n+\kappa-1$. 
Again because $\zeta_i^{-\kappa+1}=\zeta_i^{n+1}$, the last equation of \eqref{equ:powersum} holds as well.  
\end{proof}

\begin{rem}
 The system of eigenvectors obtained can be seen as a multiparameter generalization of the Fourier basis, to which they reduce for $k=2$.
\end{rem}

\section{Symmetric rectangular multiparameter eigenvalue problems revisited}\label{sec:symmmultpar2}
In section \ref{sec:symmmultpar} above we gave an ad hoc definition of symmetry for a multiparameter eigenvalue problem in the form \eqref{equ:smultpar}. While this definition suffices for the purposes of the present work, it is conceptually not satisfying and not general enough in many relevant cases. In this appendix we discuss a more general and formally more convincing definition.

Let $A_1,\dots, A_k:V\to W$ be linear operators between the $m$ dimensional vector space $V$ and the $(m+k-2)$-dimensional vector space $W$. A rectangular multiparameter eigenvalue problem is the equation
\begin{equation}
 \label{equ:smultpar2}
 \sum_{i=1}^k \lambda_i A_i v = 0
\end{equation}
in complex unknowns $(\lambda_1,\dots, \lambda_k)\neq 0$ (the multiparameter eigenvalue, defined up to a multiplicative constant) and $0\neq v\in V$ (the eigenvector).

A \emph{symmetric} rectangular multiparameter eigenvalue problem is a rectangular multiparameter eigenvalue problem together with the additional data of $k$ ``dual'' operators $B_1,\dots, B_k:V\to W^*$, such that the the bilinear forms 
\[
(u,v)\mapsto B_i(u) \cdot A_j(v)
\]
are symmetric.


Let $M(u,v)$ be the $k\times k$ matrix with $i,j$-entry $B_i(u) \cdot A_j(v)$. 
If $u$ is a solution to the multiparameter eigenvalue problem for some eigenvalue $\lambda$, then $\lambda$ is a right nullvector of $M(u,u)$. If $M(u,u)$ is of corank 1, this property determines $\lambda$ up to scale.

If $u$ and $v$ are solutions to the multiparameter eigenvalue problem for distinct eigenvalues $\lambda, \lambda'$, then both $\lambda$ and $\lambda'$ are right nullvectors of $M(u,v)$. Hence all $(k-1)\times (k-1)$ minors of $M(u,v)$ vanish. This can be considered as a joint orthogonality condition for $u^{(k-1)}, v^{(k-1)}\in S^kV$.

For each symmetric multiparameter eigenvalue problem one may define a dual problem, obtained by exchanging the role of $W$ and $W^*$, and that of the $A_i$ and $B_i$.
If $u$ is a solution to the dual multiparameter eigenvalue problem for some eigenvalue $\mu$, then $\mu$ is a left nullvector of $M(u,u)$.
If $u,v$ are solutions to the dual multiparameter eigenvalue problem with eigenvalue $\mu, \mu'$, then $\mu, \mu'$ are left nullvectors of $M(u,v)$. In particular the joint orthogonality of $u^{(k-1)}, v^{(k-1)}$ with respect to the bilinear forms defined by the $(k-1)\times (k-1)$ minors of $M(u,v)$ follows again. 

\begin{ex} 
Consider the Heine-Stieltjes equation \eqref{equ:HS}. It is a rectangular multiparameter eigenvalue problem with $A_1=L-n(n-1+\sum_j m_j)x^k, A_2=x^{k-1}, A_3=x^{k-2},\cdots, A_k=1$. 
Here $V=V_n$ and $W=V_{n+k-2}$.
The eigenvalue problem is symmetric, with the dual operators $B_i$ being defined by the scalar products:
\begin{gather*}
 B_i : V \mapsto W^* \\
v \mapsto (w\mapsto \skp{v}{w}_i ).
\end{gather*}
The dual rectangular eigenvalue problem is equation \eqref{equ:eigval}, whose solutions (i.~e., the eigenvectors) form a jointly orthogonal system.
\end{ex}

The results of this paper may be interpreted as saying that under suitable technical conditions the solutions (eigenvectors) of a symmetric rectangular eigenvalue problem are equal to those of the dual problem.

\nocite{holstshapiro}
\bibliographystyle{plain}
\bibliography{poly}

\end{document}